\documentclass[12pt,a4paper]{article}
\usepackage[utf8]{inputenc}
\usepackage{mathtools}
\usepackage{amsmath}
\usepackage{amsfonts}
\usepackage{amssymb}
\usepackage{amsthm}
\usepackage{graphicx}
\usepackage[left=3cm,right=3cm,top=2cm,bottom=2cm]{geometry}
\usepackage{hyperref}
\usepackage{orcidlink}

\usepackage{tikz}
\usetikzlibrary{decorations.pathreplacing,calligraphy}

\theoremstyle{plain}
\newtheorem{theorem}{Theorem}[section]
\newtheorem*{theorem*}{Theorem}
\newtheorem{lemma}[theorem]{Lemma}
\newtheorem{proposition}[theorem]{Proposition}

\newtheorem*{corollary*}{Corollary}

\newtheorem*{conjecture*}{Conjecture}
\theoremstyle{remark}

\newtheorem*{remark}{Remark}
\theoremstyle{definition}
\newtheorem{definition}[theorem]{Definition}
\newtheorem*{definition*}{Definition}

\newcommand{\R}{\mathbb{R}}
\newcommand{\N}{\mathbb{N}}

\newcommand{\bbP}{\mathbb{P}}
\newcommand{\E}{\mathbb{E}}
\newcommand{\eps}{\varepsilon}
\newcommand{\cC}{\mathcal{C}}
\newcommand{\cG}{\mathcal{G}}
\newcommand{\Expl}{\mathrm{Expl}}
\newcommand{\zerovec}{\mathbf{0}}
\newcommand{\bv}{\mathbf{v}}
\newcommand{\bbL}{\mathbb{L}}
\newcommand{\rmd}{\mathrm{d}}
\newcommand{\Vorder}{\vec{V}}
\newcommand{\Ghat}{\hat{G}}
\newcommand{\config}{\mathbf{N}}

\newcommand{\rup}[1]{\left\lceil #1 \right\rceil}
\newcommand{\rdown}[1]{\left\lfloor #1 \right\rfloor}

\renewcommand{\emptyset}{\varnothing}

\title{Exponential decay for the random connection model using asymptotic transitivity\thanks{This work was supported by the The Royal Society [grant number RF\textbackslash{}ERE\textbackslash{}231149].}}
\author{Frankie Higgs\thanks{\href{mailto://frankiehiggs@gmail.com}{frankiehiggs@gmail.com}}\;  \orcidlink{0000-0002-7300-8412} \\ University of Rome La Sapienza}

\begin{document}

\maketitle

\begin{abstract}
We prove that the probability the cluster of the origin
in a subcritical Poisson random connection model (RCM)
has size at least $n$ decays exponentially as $n$ increases,
under minimal assumptions.

We extend a recent method of Vanneuville from Bernoulli percolation
on vertex-transitive graphs
to the RCM,
where the vertex set has little structure.
The key idea is that the subcritical RCM can be constructed
by site percolation on a high-intensity RCM.
The latter RCM
becomes ``almost vertex-transitive''
at very high intensities.

We obtain the result for connection functions
with unbounded support,
a setting in which it was not previously known.
\end{abstract}

\section{Introduction}

Take a random countable set of points $\eta$ in Euclidean space $\R^d$ for $d \in \N$,
and create a random graph by joining every pair of points $x,y \in \eta$
with an edge with probability $g(x-y)$, independently of all other pairs,
where $g$ is a function satisfying $\lim_{|x| \to \infty} g(x) = 0$.
If $\eta$ is a Poisson point process of intensity $\lambda$,
we call this graph the \emph{random connection model} (RCM).
When $g(x) = \mathbb{1}(\|x\| \leq 1)$,
the RCM is the well-known random geometric graph, or continuum percolation model.

The RCM is a natural model of random networks in which spatial relationships are important but long-range connections are possible, and both the location of vertices and edges between them are random.

Subject to certain conditions on $g$,
there is a percolation phase transition in $\lambda$,
i.e.\ there exists a $\lambda_c \in (0,\infty)$
such that all connected components of the RCM are almost surely finite
when $\lambda < \lambda_c$,
and there exists at least one infinite connected component
when $\lambda > \lambda_c$.

One natural question is whether this percolation phase transition is continuous,
i.e.\ whether there is an infinite component when $\lambda = \lambda_c$.
With sufficiently ``spread-out''
connection functions,
this is known both for the RCM~\cite{rcm-lace-expansion}
and for discrete models of long-range percolation~\cite{power-law-critical}.
Such a result is conjectured for many short-range percolation models including the random geometric graph,
but has only been proved in certain cases, in particular when the dimension is either 2 or sufficiently high.
By changing the connection function $g$
the RCM is a family which naturally contains both short-range and long-range
percolation models.

We will prove a \emph{sharpness} result for the RCM with \emph{any} monotonic connection function: in the whole subcritical phase the clusters are not just finite but are also unlikely to be large.
For $\lambda < \lambda_c$
we show that the probability that a typical vertex
is in a connected component of size at least $n \in \N$
decays exponentially in $n$.
This was previously only known when $g$ had bounded support
(see \cite[Lemma 5.2]{sharpness-bounded-support}).
If $g$ has unbounded support then it is possible that
the cluster containing a given vertex may have a large diameter but few points,
or many points in a small area.
Theorem~\ref{thm:exponential-decay}
shows that the latter does not occur:
although the RCM can have arbitrarily long edges,
in the subcritical regime the size of clusters still has an exponential tail.

For both short- and long-range \emph{discrete} models,
it is known that an exponential decay result similar to
Theorem~\ref{thm:exponential-decay} holds whenever
the susceptibility is finite,
i.e.\ $\E|\cC_0| < \infty$
(see the ``Aizenman-Newman-Barsky property'' in \cite[Chapter 3.2]{analyticity}).
Finite susceptibility for subcritical long-range Bernoulli percolation
was proved in \cite{AV-susceptibility},
and a simple new proof appeared in \cite{DT-susceptibility}.

Two recent preprints appeared following the submission of this paper
which also prove and generalise our Theorem~\ref{thm:exponential-decay}
using different methods.
Chebunin and Last \cite{chebunin-last} independently proved a sharp phase transition result
for stationary \emph{marked} random connection models
under a moment bound on the connection function.
In the unmarked case their result generalises our Theorem~\ref{thm:exponential-decay}
to \emph{any} integrable connection function.

Caicedo and Kolesnikov \cite{caicedo-kolesnikov} independently proved
sharpness results generalising our Theorem~\ref{thm:exponential-decay}.
Their paper extends the method of \cite{drt},
which uses the OSSS inequality,
from the Poisson--Boolean model to more general marked random connection models.

In this paper,
we use the recent stochastic comparison technique of Vanneuville \cite{vanneuville},
who proved the same exponential decay result for Bernoulli bond percolation
on infinite, vertex-transitive graphs such as the integer lattice.
We construct the RCM
using Bernoulli site percolation on an RCM of a much higher intensity.
A similar idea was used for the random geometric graph in \cite{noise-sensitivity}.
The difficulty in adapting Vanneuville's method
is that \cite{vanneuville} relies
on the graph being vertex transitive,
while the higher-intensity RCM
is not transitive.
For example it can have both arbitrarily large cliques and isolated vertices.
Our new insight is that as the intensity tends to infinity the denser RCM becomes
``asymptotically vertex-transitive'':
the probabilities of events related to
the behaviour of the percolation configuration around a vertex $x$
become close to uniform in $x$.

As an example of ``asymptotic vertex-transitivity'',
suppose we are given an RCM $G_\lambda$ of intensity $\lambda$,
and we are interested in events
such as whether a given vertex $x$ is isolated,
or whether the connected component
containing $x$ is infinite. 
If we were to perform site percolation with parameter $p$ on $G_\lambda$,
then we get a \emph{random} conditional probability measure $P_p^{G_\lambda}$
which is the Bernoulli product measure on the vertices of $G_\lambda$.
Then the random function $f_{p}^{G_\lambda} : V(G_\lambda) \to [0,1]$ given by
$x \mapsto P_p^{G_\lambda}(x \text{ has no open neighbours})$
varies significantly across $V(G_\lambda)$.
For example, some vertices will already be isolated in $G_\lambda$ and have $f_p^{G_\lambda}(x) = 1$,
and some vertices will have an unusually large number of neighbours in $G_{\lambda}$
so $f_p^{G_\lambda}(x)$ will be close to 0.
But if we take a large $N \in \N$ and change the parameters
$\lambda \mapsto N\lambda$, $p \mapsto p/N$,
then the marginal distribution of the percolation configuration stays the same,
but there is a constant $c_{p\lambda} \in (0,1)$ such that for $K > 0$,
$\sup_{x \in V(G_{N\lambda}) \cap [-K,K]^d} |f_{p/N}^{G_{N\lambda}}(x) - c_{p\lambda}| \to 0$
in probability as $N \to \infty$.

We expect that this method will be useful
to transfer other approaches and results
from discrete percolation to continuum models.

\subsection{Definitions}

\begin{definition}[Random connection model]
	\label{def:rcm}
	Let $g : \R^d \to [0,1]$ and assume that 
	$g(x)$ depends only on $\|x\|$
	and that if $\|x\| \leq \|y\|$ then $g(x) \geq g(y)$.
	Let $\eta \subset \R^d$ be a discrete set of points.
	For every pair $x,y \in \eta$ with $x \not= y$,
	let $U_{x,y}$ be a uniform random variable in $[0,1]$,
	independent of all other pairs.
	Construct the random graph $G_g(\eta)$
	with vertex set $\eta$ and edge set $E_g(\eta)$
	given by $xy \in E_g(\eta)$ if and only if $U_{x,y} \leq g(x-y)$.
	We call $g$ the \emph{connection function}
	of the random graph $G_g(\eta)$.
	
	If $\eta$ is a Poisson point process,
	we call $G_g(\eta)$ the \emph{random connection model} (RCM).
	If $\eta$ has intensity $\lambda > 0$,
	let $\bbP_\lambda$ denote the law of $\eta$,
	and $\bbP_{\lambda,g}$ the joint law of $\eta$ and the uniform random variables $U_{x,y}$.\\
	
	We will often add extra points $x_1, \dots, x_k$
	to the vertex set of the RCM.
	In this case we introduce new uniform random variables $U_{x_i,y}$
	for $y \in \eta \cup \{x_1, \dots, x_k\} \setminus \{x_i\}$
	independent of those defined between vertices of $\eta$,
	and then determine the edges of $G_g(\eta \cup \{x_1, \dots, x_k\})$
	as above.
	If we add only one vertex $x \in \R^d$ it will be convenient to write
	$\eta^{x} := \eta \cup \{x\}$.
	In a slight abuse of notation we will also use $\bbP_{\lambda,g}$
	to denote the joint law of $\eta$ and the extended collection of uniforms.
\end{definition}

\begin{remark}
	We could replace the assumptions we make on $g$ with others,
	such as that $g(-x) = g(x)$ for all $x$ and $g$ is continuous.
	However, the assumptions in Definition~\ref{def:rcm} simplify our arguments.
	To use a more general class of connection functions,
	the only proof which would require significant modification
	would be that of Lemma~\ref{thm:G-coupling}:
	if we replace $g(x-y)$
	with the supremum of $g(x' - y')$
	for $x'$ and $y'$ close to $x$ and $y$,
	the probability of adding new edges is small.
\end{remark}

\begin{definition}[Volume tail]
	\label{def:size-prob}
	Given a point process $\eta$,
	let $\cC_0$ denote the component of $G_g(\eta^{\zerovec})$
	containing $\zerovec$.
	For $n \geq 0$ and $\lambda > 0$ let
	\[
	\psi_n(\lambda) := \bbP_{\lambda,g}[ |\cC_0| \geq n ].
	\]
\end{definition}

\begin{remark}
	If $\int_{\R^d}g(x)\rmd x = \infty$ then, almost surely,
	the neighbourhood of $\zerovec$ in $G_g(\eta^{\zerovec})$
	is both infinite and unbounded,
	so $\psi_n(\lambda) = 1$ for all $n \in \N$ and $\lambda > 0$.
	To make the RCM a non-trivial percolation model we assume that
	\begin{align}
		\label{eq:g-finite}
		0 < \int_{\R^d} g(x)\rmd x < \infty.
	\end{align}
	Then \cite[Theorem 6.1]{meester-roy} states that,
	if $d \geq 2$ and $g$ satisfies \eqref{eq:g-finite},
	there exists $\lambda_c \in (0,\infty)$
	such that $\cC_0$ is finite almost surely if $\lambda < \lambda_c$,
	and $\bbP_{\lambda,g}[|\cC_0| = \infty] > 0$ if $\lambda > \lambda_c$.
\end{remark}

\begin{remark}
	In the one-dimensional RCM, \eqref{eq:g-finite} is necessary but not sufficient
	to have a non-trivial phase transition.
	It is shown in \cite[Theorem 3]{1d-marked}
	that in one dimension,
	if $g$ follows a power-law distribution
	with $g(x) \sim \|x\|^{-\alpha}$ with $\alpha > 2$,
	then there is no infinite component for any $\lambda < \infty$.
	On the other hand if $\alpha \in (1,2)$
	then there is a finite $\lambda_c > 0$
	such that the RCM has an infinite component for $\lambda > \lambda_c$
	\cite{random-walk-on-rcm}
	(see their remark at the bottom of page 1051).
	
	All of our results are true when $d=1$
	for any $g$ satisfying \eqref{eq:g-finite},
	even if $\lambda_c = \infty$
	(in which case our results hold for all finite $\lambda$).
\end{remark}

\begin{remark}
	The \emph{Poisson--Boolean model} is related to the RCM.
	Instead of connecting points with edges using $g$,
	in the Poisson--Boolean model we place a ball of random radius
	centred at each point of $\eta$,
	with an iid radius for each point.
	Alternatively, we can view the Poisson--Boolean model as a marked RCM:
	mark each vertex with a random positive radius,
	and add an edge between $(x,R_x), (y,R_y) \in \R^d \times (0,\infty)$
	if and only if $R_x + R_y \geq \|x-y\|$.
	
	Gou\'er\'e \cite{gouere} proved that there is a non-trivial
	phase transition in the Poisson--Boolean model in $\R^d$
	if and only if the radius distribution has a finite $d$th moment.
	Later, Duminil-Copin, Raoufi and Tassion \cite{drt}
	proved a sharpness result:
	for sufficiently small intensities $\lambda$,
	if the radius distribution has an exponential tail, then
	the probability that the origin is connected to distance $r$
	decays exponentially in $r$.
	If the radius distribution has a heavier tail,
	they showed that the probability of reaching distance $r$
	is equivalent to the probability of there being a single large ball
	which both contains the origin and intersects $\partial B(\zerovec,r)$.
	Using the OSSS inequality, they show that if the radius distribution
	has a finite $(5d-3)$th moment then the above results hold
	for the entire subcritical regime.
	
	
	The above results concern the geometric size of $\cC_0$
	instead of the number of points it contains.
	In contrast with the RCM,
	an exponential bound such as our Theorem~\ref{thm:exponential-decay}
	on $|\cC_0|$ does not always hold
	for subcritical Poisson--Boolean models.
	If $\E R_\zerovec^d < \infty$ but $\E R_\zerovec^{2d} = \infty$,
	then $\lambda_c > 0$
	but for any $\lambda < \lambda_c$,
	$|\cC_0| < \infty$ almost surely 
	while $\E_\lambda |\cC_0| = \infty$
	(and so $\psi_n(\lambda)$ does not decay exponentially) \cite[Corollary 3.2]{meester-roy}.
\end{remark}

\subsection{Main results}

These results extend those of \cite{vanneuville}
from nearest-neighbour Bernoulli bond percolation on a transitive graph
to the random connection model.

\begin{theorem}
	\label{thm:exponential-decay}[Exponential decay]
	Let $g$ be a connection function satisfying \eqref{eq:g-finite}
	and the conditions of Definition~\ref{def:rcm}.
	For all $\lambda < \lambda_c(g)$ there exist $c, C > 0$
	depending on $\lambda$ and $g$
	such that for all $n \in \N$,
	$$
	\psi_n(\lambda) \leq C e^{-cn}.
	$$
\end{theorem}

\begin{remark}
	When the support of $g$ is bounded,
	this bound was recently shown for the RCM in \cite[Lemma 5.2]{sharpness-bounded-support}.
	As far as we can tell our Theorem~\ref{thm:exponential-decay}
	is the first time such an exponential decay result
	has been shown for the RCM
	with a long-range connection function.
\end{remark}

To prove Theorem~\ref{thm:exponential-decay},
following Vanneuville~\cite{vanneuville},
we introduce a ``ghost field''
as an extra source of randomness.
The idea is that when the ghost field has a low intensity,
it only intersects large components.

\begin{definition}[Ghost field]
	\label{def:ghost}
	Let $h > 0$.
	Given a graph $G = (V,E)$,
	let $\cG \subseteq V$ be a random subset of the vertices,
	where for each $v \in V$,
	we have $v \in \cG$ with probability $1 - e^{-h}$
	independently of the other vertices.
	Call $v \in V$ \emph{green} if $v \in \cG$.
	
	For the graph $G = G_g(\eta^{\zerovec})$,
	let $\bbP_{\lambda,g,h}$ denote the joint law of
	$\eta$, $(U_{x,y})_{x,y \in \eta^\zerovec}$
	and $\cG$. Then let
	$$
	m_h(\lambda) := \bbP_{\lambda,g,h}[ \cC_0 \cap \cG \not= \emptyset ].
	$$
\end{definition}

Theorem~\ref{thm:exponential-decay}
can be derived immediately from the following theorem.
\begin{theorem}[Reducing $\lambda$ slightly]
	\label{thm:decrease-lambda}
	Let $g$ be as above.
	Let $\lambda \in (0,\infty)$ and $h \in (0,\infty)$,
	and let $\lambda' = \lambda(1 - m_h(\lambda))$.
	Then for all $n \in \N$,
	\[
	\psi_n(\lambda') \leq \frac{1}{1 - m_h(\lambda)} \psi_n(\lambda) e^{-hn}.
	\]
\end{theorem}



\section{Proof of main results}
\label{sec:bounded}

Section~\ref{sec:comparison}
closely follows the \emph{stochastic comparison} method of \cite{vanneuville},
which applies to Bernoulli percolation on a vertex-transitive graph,
to the random connection model.
The idea behind this method for Bernoulli percolation
is that conditioning on certain decreasing events
has less of an effect than decreasing the percolation parameter.
To apply it to our continuum model,
we use the observation that the RCM can be constructed by site percolation
on an RCM with the same connection function
but where the point process has much higher intensity.

Section~\ref{sec:transitivity} is where we depart from \cite{vanneuville}
and develop the tools necessary
to apply Vanneuville's method to
the random connection model.
In particular we prove ``asymptotic vertex-transitivity'':
first we take an RCM with intensity $N\lambda$,
and construct the intensity-$\lambda$ RCM by site percolation
with parameter $1/N$.
We show that with high probability when $N$ is large,
the conditional probability
that a vertex $x$
in the higher-intensity graph
is connected to the ghost field
by open vertices is bounded
uniformly in $x$
by $m_h(\lambda) + \eps$.

\subsection{Stochastic comparison}
\label{sec:comparison}

Since we place a vertex at the origin in our random connection model,
it will be useful to define site percolation clusters for \emph{rooted} graphs
in such a way that we do not close the root vertex.

\begin{definition}[Percolation on rooted graphs]
	\label{def:percolation}
	Let $G = (V,E)$ be a finite graph
	with root $\rho \in V$.
	For $p \in (0,1)$ a \emph{percolation configuration}
	on $G$ is a random variable $\omega \in \{0,1\}^{V \setminus \{\rho\}}$
	such that $\omega(x) = 1$ with probability $p$
	for each $x \in V \setminus \{\rho\}$,
	independently of all other vertices.
	Let $P^G_p$ denote the law of $\omega$.
	If $\cG$ is a field of independent $\mathrm{Bernoulli}(1 - e^{-h})$ random variables
	indexed by $V$,
	let the joint law of $\omega$ and $\cG$ be $P^G_{p,h}$.
	
	We will frequently identify $\omega$
	with the set $\{ x \in V : \omega(x) = 1 \}$.
	We then talk about the induced subgraph of $G$
	with vertex set $\omega \cup \{ \rho \}$.
\end{definition}

\begin{definition}[Pivotal vertices]
	\label{def:pivotal}
	
	For $\omega \in \{0,1\}^{V\setminus\{\rho\}}$, $x \in V \setminus \{\rho\}$
	and $i \in \{0,1\}$,
	define $\omega_x^i \in \{0,1\}^{V\setminus\{\rho\}}$
	by $\omega_x^i(x) = i$
	and $\omega_x^i(y) = \omega(y)$ for $y \not= x$.
	
	Let $A \subseteq \{0,1\}^{V \setminus \{\rho\}} \times \{0,1\}^{V}$
	be an event concerning $\omega$ and $\mathcal{G}$.
	Given $\omega$ and $\mathcal{G}$, we say that $x \in V \setminus \{\rho\}$
	is \emph{pivotal} for $A$
	if $\mathbb{1}_A(\omega^0_x, \mathcal{G}) \not=
	\mathbb{1}_A(\omega^1_x, \mathcal{G})$.
\end{definition}

\begin{definition}[Percolation on high intensity RCMs]
	\label{def:omega}
	Let $\lambda > 0$ and $N \in \N$.
	Given a homogeneous Poisson point process $\eta$
	with intensity $N \lambda$,
	let $\omega_N \in \{0,1\}^{\eta}$
	be a percolation configuration on $G_g(\eta^\zerovec)$
	with parameter $1/N$
	where $\zerovec$ is the root.
	Let $P_N$ denote the conditional law of $\omega_N$
	given $\eta$,
	and let $P_{N,h}$ denote the joint conditional law
	of $\omega_N$ and $\mathcal{G}$ given $\eta$.
\end{definition}

\begin{remark}
	Note that for all $\lambda > 0$ and $N \in \N$,
	for any measurable $f$,
	\[
	\E_{\lambda,g,h}(f) = \E_{N\lambda,g}( E_{N,h}(f) ),
	\]
	where $E_{N,h}(f) := \int f\, \rmd P_{N,h}$.
\end{remark}

We adapt the notion of exploration in \cite{vanneuville}
from bond to site percolation:
\begin{definition}[Explorations]
	\label{def:exploration}
	Let $G = (V,E)$ be a finite graph
	with root $\rho \in V$
	and let $\Vorder$ be the set of orderings
	$\vec{v} = (\rho, v_1, \dots, v_{|V|-1})$
	of $V$
	which begin with the root.
	An \emph{exploration} of $V$ is a map
	\[
	\bv : \{0,1\}^{V \setminus \{\rho\}} \to \Vorder,
	\]
	\[
	\omega \mapsto (\rho, \bv_1, \dots, \bv_{|V|-1}),
	\]
	such that $\bv_1$ does not depend on $\omega$
	and for each $k$, $\bv_{k+1}$ only depends on $(\bv_1, \dots, \bv_k)$
	and $(\omega_{\bv_1}, \dots, \omega_{\bv_k})$.
	
	For a given exploration $\bv$,
	for $(\vec{v},\sigma) \in \Vorder \times \{0,1\}^{V\setminus \{\rho\}}$
	and $k \in \{0,\dots,|V|-1\}$,
	let
	$$\Expl_k(\vec{v},\sigma)
	:= \{ \omega \in \{0,1\}^{V\setminus \{\rho\}} :
	\forall j \in \{1,\dots,k\}, \bv_j = v_j \text{ and } \omega_{v_j} = \sigma_j
	\}.$$
	If $\omega$ is a random element of $\{0,1\}^{V\setminus\{\rho\}}$,
	then $\Expl_k(\vec{v},\sigma)$
	is the event that the first $k$ steps of $(\bv,\omega)$ coincide with $(\vec{v},\sigma)$.
	Let $\Expl(\vec{v},\sigma) = \Expl_{|V|-1}(\vec{v},\sigma).$
\end{definition}

\begin{definition}[Increasing events and stochastic domination]
	Given a set $S$,
	an event $A \subseteq \{0,1\}^S$
	is said to be \emph{increasing}
	if, for every $\omega \in A$,
	whenever $\omega' \in \{0,1\}^S$ satisfies
	$\omega'(x) \geq \omega(x)$ for all $x \in S$,
	we also have $\omega' \in A$.
	
	An event $A$ related to the random connection model $G_g(\eta^\zerovec)$
	is said to be increasing
	if, when $G_g(\eta^\zerovec) \in A$
	and $H$ is a graph containing $G_g(\eta^\zerovec)$
	as a subgraph,
	then $H \in A$,
	i.e.\ when $A$ occurs, adding vertices or edges to the graph
	does not stop $A$ from occurring.
	
	Given two measures $\mu$, $\nu$ either on a product space $\{0,1\}^S$
	or the space corresponding to the RCM,
	we say that $\nu$ \emph{stochastically dominates} $\mu$,
	$\mu \preceq \nu$,
	if $\mu(A) \leq \nu(A)$ for all increasing events $A$.
\end{definition}

The following lemma is adapted from \cite{vanneuville}.

\begin{lemma}[Stochastic comparison for site percolation]
	\label{thm:pivotal}
	Let $G = (V,E)$ be a finite graph with root $\rho \in V$.
	Let $p \in (0,1)$ and $h > 0$.
	Let $A \subseteq \{0,1\}^{V\setminus\{\rho\}} \times \{0,1\}^V$
	be any non-empty set
	and let $\eps \in [0,1]$.
	Assume that for every
	$(\vec{v},\sigma) \in \Vorder \times \{0,1\}^{V\setminus\{\rho\}}$
	such that $A \cap \Expl(\vec{v},\sigma) \not= \emptyset$ we have
	$$
	\forall k \in \{0, \dots, |V|-2\}, \quad
	P_{p,h}^G[v_{k+1} \text{ is pivotal for } A \,|\, A \cap \Expl_k(\vec{v},\sigma)] \leq \eps.
	$$
	Then
	$$
	P_{p(1-\eps)}^G \preceq P_{p,h}^G[\omega \in \cdot \,|\, A].
	$$
\end{lemma}
\begin{proof}
	This involves little more than
	verifying that the arguments in \cite{vanneuville}
	also work for site percolation.
\end{proof}

We will now apply this lemma to a particular graph:
the RCM inside a large box.

\begin{definition}
	\label{def:box}
	For $r > 0$ and $x \in \mathbb{R}^d$ let
	\[\Lambda_r(x) := x + [-r/2,r/2)^d,\]
	and $\Lambda_r := \Lambda_r(\zerovec)$.
	
	Fix $R > 0$ and let $n \in \mathbb{N}, \lambda > 0, h > 0$.
	For a Poisson point process $\eta$
	of intensity $\lambda$ in $\R^d$,
	let $G^{nR}_g = G_g^{nR}(\eta^\zerovec) := G_g(\eta^\zerovec \cap \Lambda_{nR})$.
	Let $\bbP^n_{\lambda,g}$
	denote the law of $G_g^{nR}( \eta^\zerovec )$,
	and let
	$\bbP_{\lambda,g,h}^n$ denote the joint law of
	$G^{nR}_g$
	and the restriction of $\mathcal{G}$ to $\eta^\zerovec \cap \Lambda_{nR}$.
	
	For $N \geq 1$,
	let $\eta_N$ be a Poisson point process of intensity $N\lambda$ in $\Lambda_{nR}$,
	and let $G = G_g(\eta_N^\zerovec)$ with root $\zerovec$.
	Let $\omega_N$ be a site percolation configuration on $G$
	with parameter $1/N$.
	Note that $G_g(\omega_N \cup \{\zerovec\})$ has marginal law $\bbP^n_{\lambda,g}$.
\end{definition}

\begin{lemma}[Stochastic comparison for the RCM]
	\label{thm:disjoint-condition}
	Let $\lambda > 0$ and $h > 0$.
	Suppose $g$ is a connection function with support contained in $\Lambda_{R}$
	for some $R > 0$.
	Let $\lambda' = \lambda(1 - m_h(\lambda))$.
	For $G_g^{nR}$ as in Definition~\ref{def:box}
	based on a Poisson point process of intensity $\lambda$,
	let $\cC_0^n$ denote the connected component of $\zerovec$ in $G_g^{nR}$.
	Then for all $n \geq 1$,
	$$
	\bbP^n_{\lambda',g} \preceq
	\bbP^n_{\lambda,g,h}[G^{nR}_g \in \cdot
	\,|\, (\cC_0^n \cap \Lambda_{nR}) \cap \cG = \emptyset].
	$$
\end{lemma}

We will defer one part of the proof of Lemma~\ref{thm:disjoint-condition}
to the next section.
After the proof of Lemma~\ref{thm:disjoint-condition},
we will conclude this section by deriving Theorem~\ref{thm:decrease-lambda}.

\begin{proof}[Proof of Lemma~\ref{thm:disjoint-condition}]
	Fix $n \geq 1$ and let $\delta \in (0,1-m_h(\lambda))$.
	
	We will apply Lemma~\ref{thm:pivotal} to the graph $G$
	from Definition~\ref{def:box} with
	$A = \{ \cC_0^n \cap \cG = \emptyset\}$
	and $\eps = m_h(\lambda) + \delta$.
	This follows the proof of \cite[Lemma 8]{vanneuville}
	very closely,
	but to illustrate where the differences between our setting
	and that of \cite{vanneuville} become important,
	we will write out the argument.
	Let $V$ be the vertex set of $G$ and fix an arbitrary ordering of $V$
	in which $\zerovec$ comes first.
	Define an exploration of $V$ as follows:
	at each step, among the vertices connected to $\zerovec$
	either directly or via a path of open vertices we have already explored,
	reveal the state of the one with the lowest index.
	Continue this process until there are no more such vertices,
	at which point we have revealed $\cC_0^n$ in $G_g(\omega_N \cup \{\zerovec\})$.
	Reveal the states of the remaining vertices
	in the order we fixed earlier.
	
	Let $(\vec{v},\sigma) \in \Vorder \times \{0,1\}^{V\setminus\{\zerovec\}}$
	such that $\{\cC_0^n\cap\cG = \emptyset \} \cap \Expl(\vec{v},\sigma) \not= \emptyset$.
	If $\Expl_k(\vec{v},\sigma)$ holds
	and $v_{k+1}$ is pivotal for $\{\cC_0^n \cap \cG = \emptyset\}$,
	then this implies that $v_{k+1}$ is connected
	to a green vertex
	by a path of open vertices, none of which belong to
	$\{\zerovec, v_1, v_2, \dots, v_{k}\}$.
	Denote this event by $B_{k+1}$.
	By the Harris-FKG inequality applied to
	the product measure $P^G_{1/N,h}[ \,\cdot \,|\, \Expl_k(\vec{v},\sigma)]$,
	the increasing event $B_{k+1}$,
	and the decreasing event $\{\cC_0^n \cap \cG = \emptyset\}$,
	we have
	\begin{align*}
		P^G_{1/N,h}[ v_{k+1} \text{ is }&\text{pivotal for } \{ \cC_0^n \cap \cG = \emptyset \}
		\,|\, \cC_0^n \cap \cG = \emptyset, \Expl_k(\vec{v},\sigma) ]\\
		&\leq
		P^G_{1/N,h}[ B_{k+1} \,|\, \cC_0^n \cap \cG = \emptyset, \Expl_k(\vec{v},\sigma) ]\\
		&\leq
		P^G_{1/N,h}[ B_{k+1} \,|\, \Expl_k(\vec{v},\sigma) ].
	\end{align*}
	Then, because $B_{k+1}$ depends only on
	$V \setminus \{\zerovec,v_1,\dots,v_{k}\}$,
	$P^G_{1/N,h}[ B_{k+1} \,|\, \Expl_k(\vec{v},\sigma) ] = P^G_{1/N,h}[B_{k+1}]$.
	
	Now we differ from \cite{vanneuville}
	since $G$ is a random graph.
	Let
	\begin{align}
		\label{eq:atypical-event}
		E_{\delta,N} := \left\{
		\max_{k \in \{0,\dots,|V|-2\}} P^G_{1/N,h}[B_{k+1}] \leq m_h(\lambda) + \delta
		\right\},
	\end{align}
	and then conditional on $E_{\delta,N}$,
	the conditions of
	Lemma~\ref{thm:pivotal} are met so
	$$P^G_{\frac{1}{N}(1 - m_h(\lambda) - \delta)} \preceq
	P^G_{1/N,h}\left(G^{nR}_g \in \cdot
	\,|\, (\cC_0^n \cap \Lambda_{nR}) \cap \cG = \emptyset\right).$$
	We claim that $\bbP_{N\lambda,g}(E_{\delta,N}) \to 1$ as $N \to \infty$.
	We defer the proof of this claim until Proposition~\ref{thm:transitivity}
	in the next section.
	A related claim which will also be useful here is that
	the variance of $P^{G}_{1/N,h}((\cC_0^n \cap \Lambda_{nR}) \cap \cG = \emptyset)$
	tends to zero as $N \to \infty$
	(we will not prove this latter claim;
	it is simpler than the uniform-in-$k$ bound of \eqref{eq:atypical-event}
	and follows by a similar argument to \cite[Lemma 5.4]{ATT-planar}).
	To see that the claims together imply our result,
	let $A$ be an increasing event
	in the $\sigma$-algebra corresponding to $\bbP^n_{\lambda,g}$,
	and let $B = \{(\cC_0^n \cap \Lambda_{nR}) \cap \cG = \emptyset\}$.
	
	Note that by the second claim, for sufficiently large $N$ we have
	$$\bbP_{N\lambda,g}\left(\frac{1}{P_{1/N,h}(B)} \leq \frac{1}{\bbP^n_{\lambda,g,h}(B)} + \delta \right) \geq 1 - \delta.$$
	Therefore
	\begin{align*}
		\E_{N\lambda,g}^n&( P^G_{1/N,h}[ G_g^{nR} \in A \,|\, B ] )\\
		&\leq
		\E_{N\lambda,g}^n\left( P^G_{1/N,h}[G_g^{nR} \in A \,|\, B]
		\mathbb{1}\!\left\{
		\frac{1}{P_{1/N,h}(B)}\leq\frac{1}{\bbP^n_{\lambda,g,h}(B)} + \delta
		\right\}
		\right) + \delta\\
		&=
		\E_{N\lambda,g}^n\left( \frac{P^G_{1/N,h}[(G_g^{nR} \in A) \cap B]}{P_{1/N,h}(B)}
		\mathbb{1}\!\left\{
		\frac{1}{P_{1/N,h}(B)}\leq\frac{1}{\bbP^n_{\lambda,g,h}(B)} + \delta
		\right\}
		\right) + \delta\\
		&\leq
		\E_{N\lambda,g}^n\left( \frac{P^G_{1/N,h}[(G_g^{nR} \in A) \cap B]}{\bbP^n_{\lambda,g,h}(B)}
		\right) + 2\delta\\
		&=
		\bbP^n_{\lambda,g,h}[ G_g^{nR} \in A \,|\, B ] + 2\delta.
	\end{align*}
	Then
	\begin{align*}
		\bbP^n_{\lambda,g,h}[G^{nR}_g \in A \,|\, B]
		&\geq
		\E^n_{N\lambda,g}\!\left( P^G_{1/N,h}[G^{nR}_g \in A \,|\, B ] \right) - 2\delta\\
		&\geq
		\E^n_{N\lambda,g}\!\left( P^G_{1/N,h}[G^{nR}_g \in A \,|\, B ] \times \mathbb{1}_{E_{\delta,N}} \right) - 2\delta\\
		&\geq
		\E^n_{N\lambda,g}\!\left( P^G_{\frac{1}{N}(1 - m_h(\lambda) - \delta)}[G^{nR}_g \in A] \times \mathbb{1}_{E_{\delta,N}} \right) - 2\delta\\
		&\geq
		\E^n_{N\lambda,g}\!\left( P^G_{\frac{1}{N}(1 - m_h(\lambda) - \delta)}[G^{nR}_g \in A] \right) - \bbP_{N\lambda,g}(E_{\delta,N}^c) - 2\delta\\
		&=
		\bbP^n_{\lambda(1 - m_h(\lambda)-\delta),g}\!\left( G^{nR}_g \in A \right) - \bbP_{N\lambda,g}(E_{\delta,N}^c) - 2\delta.
	\end{align*}
	We can take $N \to \infty$ in the above inequality and get
	$$
	\bbP^n_{\lambda,g,h}[G^{nR}_g \in A \,|\, B]
	\geq
	\bbP^n_{\lambda(1 - m_h(\lambda)-\delta),g}\!\left( G^{nR}_g \in A \right) - 2\delta,
	$$
	and then note that since $\Lambda_{nR}$ is bounded,
	$\lambda \mapsto \bbP^n_{\lambda,g}\!\left( G^{nR}_g \in A \right)$
	is continuous.
	Therefore we can take the limit $\delta \downarrow 0$
	to conclude
	$$
	\bbP^n_{\lambda,g,h}[G^{nR}_g \in A \,|\, (\cC_0^n \cap \Lambda_{nR}) \cap \cG = \emptyset]
	\geq
	\bbP^n_{\lambda(1 - m_h(\lambda)),g}\!\left( G^{nR}_g \in A \right),
	$$
	as claimed.
\end{proof}

\begin{proof}[Proof of Theorem~\ref{thm:decrease-lambda}]
	First assume that $g$ has support
	contained in $\Lambda_R$ for some $R > 0$.
	Then Theorem 1.10 follows from Lemma~\ref{thm:disjoint-condition}
	exactly as the equivalent lemma in \cite{vanneuville}
	implies the equivalent theorem for Bernoulli percolation,
	noting that the event $\{|\cC_0| \geq n\}$ depends only on $G^{nR}_{g}$,
	so is equal to $\{|\cC_0^n| \geq n\}$.
	
	Now suppose $g$ has unbounded support,
	and for $R > 0$ and $x \in \mathbb{R}^d$
	let $g^R(x) := g(x)\mathbb{1}(\|x\| \leq R)$.
	Let
	$\psi_n^R(\lambda) := \mathbb{P}_{\lambda,g^R}[ |\cC_0| \geq n ]$
	and $m_h^R(\lambda) := \mathbb{P}_{\lambda,g^R,h}[\cC_0 \cap \cG \not= \emptyset]$.
	Note that as $R \to \infty$, $\psi_n^R(\lambda) \uparrow \psi_n(\lambda)$
	and $m_h^R(\lambda) \uparrow m_h(\lambda)$.
	Let $\lambda_R' = \lambda(1-m_h^R(\lambda))$.
	Since $\lambda_R' \geq \lambda'$
	and $\psi_n^R$ is non-decreasing,
	we have
	$\psi_n(\lambda') = \lim_{R \to \infty}\psi_n^R(\lambda')
	\leq \limsup_{R \to \infty} \psi_n^R(\lambda_R')$.
	Applying Theorem 1.10 with the finite-range connection function $g^R$ gives
	$\psi_n^R(\lambda_R') \leq \frac{1}{1 - m_h^R(\lambda)}\psi_n^R(\lambda)e^{-hn}$,
	and so
	\begin{align*}
		\psi_n(\lambda')
		&\leq \limsup_{R \to \infty}\frac{1}{1 - m_h^R(\lambda)}\psi_n^R(\lambda)e^{-hn}
		= \frac{1}{1 - m_h(\lambda)}\psi_n(\lambda)e^{-hn}
	\end{align*}
	as required.
\end{proof}

\begin{proof}[Proof of Theorem~\ref{thm:exponential-decay}]
	Theorem~\ref{thm:exponential-decay}
	can be derived from
	Theorem~\ref{thm:decrease-lambda}
	by the same argument used in \cite{vanneuville} to derive Theorem 1 from Theorem 2.
\end{proof}

\subsection{Asymptotic vertex-transitivity}
\label{sec:transitivity}

The main result of this section, Proposition~\ref{thm:transitivity},
completes the proof of Lemma~\ref{thm:disjoint-condition}
by showing that the event \eqref{eq:atypical-event}
has probability close to 1.

This is an example of the ``asymptotic transitivity''
property of the high-intensity RCM.
It says that with high probability, there are no
vertices which find it unusually easy to connect to the ghost field.
Similar statements
would be true if we replaced $B_{k+1}$ in \eqref{eq:atypical-event}
by any increasing event $A_x$
which depends only on the graph structure of an RCM in a neighbourhood of $x$.
Dealing carefully with boundary effects,
similar arguments could also establish a lower bound on the conditional
(or quenched) probabilities $P_{1/N,h}^G(A_x)$.

The method is therefore very flexible,
and could be applied to many other percolation-type
problems for the random connection model
and related models.

\begin{proposition}[No atypical vertices]
	\label{thm:transitivity}
	Let $K \in \N$,
	$\lambda \in (0,\infty)$,
	and $h > 0$.
	Recall from Definition~\ref{def:box} that $\eta_N$
	denotes a Poisson point process of intensity $N\lambda$ in $\Lambda_{K}$
	and $\omega_N$ a percolation configuration on $\eta_N^\zerovec$.
	Let $\delta > 0$, then
	$$
	\bbP_{N\lambda,g}\!\left(
	\sup_{x \in \eta_N^\zerovec} P_{1/N,h}\!\left[
	\parbox{4.2cm}{\rm\centering $x$ is connected to a green vertex in $G_g(\omega_N \cup \{x\})$}
	\right] > m_h(\lambda) + \delta
	\right) \to 0
	$$
	as $N \to \infty$.
\end{proposition}

Our proofs involve discretising the RCM,
so first we introduce some new notation.

\begin{definition}[Decomposition of the large box]
	\label{def:grid}
	For $K > 0$ and $s > 0$,
	let $\bbL_{K,s}$ be the
	rescaled lattice $s \mathbb{Z}^d \cap \overline{\Lambda_{K-s}}$.
	Recall the definition of the boxes $\Lambda_s(x)$
	from Definition~\ref{def:box}.
	Assume $K/s \in \N$,
	then $\Lambda_K$ is partitioned by $(\Lambda_s(x) : x \in \bbL_{K,s})$.
	Given such $K$ and $s$,
	for $x \in \Lambda_K$ let $\hat{x}$ denote the unique element of $\bbL_{K,s}$
	such that $x \in \Lambda_s(\hat{x})$.
\end{definition}

\begin{definition}[Discretised connection function]
	Let $K > 0$ and let $s > 0$ be such that $K/s \in \N$.
	Let $g$ be a connection function,
	and for distinct $x,y \in \Lambda_K$,
	define
	\begin{align}
		\label{eq:g-hat}
		\hat{g}_s(x,y) :=
		\sup\{g(p-q) : p \in \Lambda_s(\hat{x}), q \in \Lambda_s(\hat{y})\}.
	\end{align}
	For the same uniform random variables $U_{x,y}$ used to construct $G_g(\eta)$,
	define a new graph $\Ghat_s(\eta)$ on vertex set $\eta$
	by including the edge $xy$ if and only if
	$U_{x,y} \leq \hat{g}_s(x,y)$.
\end{definition}

\begin{remark}
	This is similar to the discretised version
	of the random geometric graph
	used in \cite[Section 7]{noise-sensitivity},
	except that we have left the vertex set unchanged.
\end{remark}


\begin{definition}[Discretised ghost field]
	Let $K$ and $s$ be as above and $\lambda > 0$.
	Let $\eta$ be a Poisson process of intensity $\lambda$ inside $\Lambda_K$.
	Let $h > 0$ and
	let $\mathcal{G} \in \{0,1\}^{\eta^\zerovec}$ be a ghost field of intensity $1-e^{-h}$
	as in Definition~\ref{def:ghost}.
	Then let
	\begin{align*}
		\hat{m}^{K,s}_h(\lambda)
		:=
		\bbP_{\lambda,g,h}[\text{the component containing $\zerovec$
			in $\hat{G}_s(\eta^\zerovec)$
			intersects $\mathcal{G}$}].
	\end{align*}
\end{definition}

Since we used the same vertex set
and the same uniform random variables
to construct them,
$G(\eta)$ and $\Ghat_s(\eta)$
are naturally coupled.
If $s$ is small,
the RCM and its discretised version are equal with high probability.

\begin{lemma}[Coupling with discretised model]
	\label{thm:G-coupling}
	Let $K > 0$ and $\lambda > 0$.
	Let $\eta$ be a Poisson point process of intensity $\lambda$ on $\Lambda_K$.
	There exists a constant $C = C(K,d)$
	such that for all $s \in \{K/2, K/3, \dots\}$,
	$$
	\bbP_{\lambda,g}[ \Ghat_s(\eta^\zerovec) \not= G_g(\eta^\zerovec) ]
	\leq C (\lambda^2 + \lambda) s.
	$$
\end{lemma}

\begin{proof}
	Note that $\Ghat_s(\eta^\zerovec) \not= G_g(\eta^\zerovec)$
	if and only if there exist distinct
	$x,y \in \eta^\zerovec$ such that $U_{x,y} \in (g(x-y),\hat{g}_s(x,y)]$.
	The expected number of such pairs of vertices is
	$$
	\frac{1}{2} \E_{\lambda,g}\!\! \left[ \sum_{(x,y) \in (\eta^\zerovec)_{\not=}^2}
	\!\!1( U_{x,y} \in (g(x-y),\hat{g}_s(x,y)]) \right]\!\!
	=
	\frac{1}{2} \E_{\lambda}\!\!\left[ \sum_{(x,y) \in (\eta^\zerovec)_{\not=}^2}
	\!\!( \hat{g}_s(x,y) - g(x-y) )
	\right]\!\!.
	$$
	By the Mecke equation,
	the right-hand side is equal to
	\begin{align*}
		\frac{\lambda^2}{2} \int_{\Lambda_K}\int_{\Lambda_K}
		( \hat{g}_s(x,y) &- g(x-y) )
		\rmd x \rmd y
		+
		\lambda \int_{\Lambda_K} ( \hat{g}_s(\zerovec,y) - g(y) ) \rmd y\\
		&\leq
		\left(\frac{\lambda^2}{2} K^d + \lambda \right) \int_{B(\zerovec,\sqrt{d}K)}
		( \hat{g}_s(\zerovec,y) - g(y) )
		\rmd y.
	\end{align*}
	Since the diameter of the boxes $\Lambda_s(x)$ is $\sqrt{d}s$
	and $g(y)$ is a decreasing function of $\|y\|$,
	for any $y \in \R^d \setminus B(\zerovec,2\sqrt{d}s)$
	we have $\hat{g}_s(\zerovec,y) \leq g(\|y\| - 2\sqrt{d}s)$.
	Therefore,
	using spherical coordinates,
	\begin{align*}
		\int_{B(\zerovec,\sqrt{d}K)}
		( \hat{g}_s(\zerovec,y) - g(y) )
		\rmd y
		&\leq
		\theta_d (2\sqrt{d}s)^{d}
		+
		d\theta_d \int_{2\sqrt{d}s}^{\sqrt{d}K} r^{d-1}(g(r - 2\sqrt{d}s) - g(r))\rmd r,
	\end{align*}
	where $\theta_d$ is the volume of the unit ball in $\R^d$.
	The latter integral can be separated and written
	\begin{align*}
		&\int_{0}^{\sqrt{d}(K - 2s)} (r + 2\sqrt{d}s)^{d-1} g(r)\rmd r
		-
		\int_{2\sqrt{d}s}^{\sqrt{d}K} r^{d-1} g(r)\rmd r \\
		&\leq
		\int_0^{2\sqrt{d}s} (r+2\sqrt{d}s)^{d-1}g(r)\rmd r
		+
		\int_{2\sqrt{d}s}^{\sqrt{d}(K - 2s)} ((r+2\sqrt{d}s)^{d-1} - r^{d-1}) g(r) \rmd r.
	\end{align*}
	The first integral is $O(s^d)$,
	and since $\int_0^{\infty}g(r)\rmd r < \infty$,
	the second integral is bounded by a constant multiple of
	$$\sup_{r \in [2\sqrt{d}s,\sqrt{d}(K-2s)]} ( (r+2\sqrt{d}s)^{d-1} - r^{d-1})
	\leq
	(d-1)(\sqrt{d}K)^{d-2} 2\sqrt{d} s + O(s^2),$$
	where the constant in the $O(s^2)$ term depends only on $K$ and $d$.
	Therefore the expected number of edges in $\hat{G}_s(\eta)$
	which are not in $G_g(\eta)$ is $O((\lambda^2+\lambda) s)$,
	so by Markov's inequality the two graphs are equal
	with probability $1 - O((\lambda^2+\lambda) s)$,
	as claimed.
\end{proof}

\begin{remark}
	The proof of Lemma~\ref{thm:G-coupling}
	is the only place we use the fact that $g(\|x\|)$
	is non-increasing in $\|x\|$.
	The lemma should be true under weaker assumptions,
	but the definition of $\Ghat$ may need to be modified.
	The proof does not need many changes if we assume
	that $g$ is continuous.
\end{remark}

We will estimate the number of vertices $\eta_N$
in each of the small boxes partitioning $\Lambda_K$,
so first we introduce some general bounds and notation.

\begin{lemma}[Poisson approximation]
	\label{thm:le-cam}
	Let $S$ be a $\mathrm{Binomial}(n,p)$ random variable,
	and let $T \sim \mathrm{Poisson}(np)$.
	Then the total variation distance between $S$ and $T$ is bounded by $9 p$.
\end{lemma}

\begin{proof}
	Follows from the main result of \cite{le-cam}.
\end{proof}

\begin{definition}[Poisson probability mass function]
	\label{def:poisson-pmf}
	For $\lambda > 0$ and $n \in \N_0$, let
	\begin{align*}
		p_{\lambda}(n) := \mathbb{P}(\mathrm{Poisson}(\lambda) = n)
		= \frac{\lambda^n}{n!}e^{-\lambda}.
	\end{align*}
	For a finite set $I$ and $\mathbf{n} = (n_i : i \in I) \in \N_0^{I}$,
	let
	$$p_{\lambda}(\mathbf{n}) := \prod_{i \in I} p_{\lambda}(n_i).$$
\end{definition}

\begin{definition}[Typical boxes]
	\label{def:R-event}
	For $n \in \N$, let $[n] := \{1,\dots, n\}$.
	Let $\lambda > 0$ and $N \in \N$.
	Given $K$ and $s$ as in Definition~\ref{def:grid}
	and a Poisson point process $\eta_N$ of intensity $N\lambda$,
	for each $z \in \bbL_{K,s}$
	let $$H_{N\lambda}(z) := | \eta_N \cap \Lambda_s(z) |.$$
	Given $\alpha \in (0,1/2)$,
	define the event
	\[
	R_{s,\alpha,N}
	:= \left\{
	(1-\alpha) \lambda s^d N < H_{N\lambda}(z) < (1+\alpha) \lambda s^d N
	\text{ for all }
	z \in \bbL_{K,s}
	\right\}.
	\]
	Given a percolation configuration $\omega$ on $\eta_N^\zerovec$,
	for $z \in \bbL_{K,s}$ let $$\Omega(z) := |\omega \cap \Lambda_s(z)|,$$
	and let $\Omega := (\Omega(z) : z \in \bbL_{K,s}) \in \N_0^{\bbL_{K,s}}$.
\end{definition}

\begin{lemma}[Large deviations for the number of vertices in a box]
	\label{thm:large-deviations}
	For $K, s, \lambda$ and $\alpha$ as above,
	there exists an $N_0 = N_0(s,\lambda,\alpha) \in \N$
	such that for all $N \geq N_0$,
	\begin{align}
		\label{eq:R-prob}
		\bbP_{N\lambda}[ R_{s,\alpha,N} ]
		&\geq
		1 - 2(K/s)^d \exp\left( -\frac{\lambda s^d \alpha^2}{24} N \right).
	\end{align}
	Moreover,
	there exists a constant $C = C(K,s,\lambda,d) > 0$
	such that on the event $R_{s,\alpha,N}$ we almost surely have
	\begin{align}
		\label{eq:config-prob}
		P_N( \Omega = \config )
		\leq
		e^{\alpha \lambda K^d} (1+\alpha)^{\sum_{z \in \bbL_{K,s}} \config(z) }
		p_{\lambda s^d}(\config) + \frac{C}{N},
	\end{align}
	for all $\config \in \N_0^{\bbL_{K,s}}$.
\end{lemma}

\begin{remark}
	If we consider only $\config \in \{0,1\}^{\bbL_{K,s}}$,
	a lower bound similar to \eqref{eq:config-prob}
	follows from the same arguments.
\end{remark}

\begin{proof}
	For any $z \in \bbL_{K,s}$,
	$H_{N\lambda}(z) \sim \mathrm{Poisson}(N\lambda s^d)$,
	and so it is equal in distribution to
	the sum of $N$ iid $\mathrm{Poisson}(\lambda s^d)$ random variables.
	By Cram\'{e}r's theorem
	\cite[Theorem I.4]{large-deviations},
	for $\alpha \in (0,1/2)$,
	$$
	\lim_{N \to \infty}\frac{1}{N}\log\mathbb{P}_{N\lambda}[ H_{N\lambda}(z) \geq (1+\alpha)\lambda s^d N ]
	= -\lambda s^d[ (1+\alpha)\log(1+\alpha) - \alpha]
	\leq -\frac{1}{4} \lambda s^d \alpha^2,
	$$
	and
	$$
	\lim_{N \to \infty}\frac{1}{N}\log\mathbb{P}_{N\lambda}[ H_{N\lambda}(z) \leq (1-\alpha)\lambda s^d N ]
	= -\lambda s^d[\alpha + (1 - \alpha)\log(1-\alpha)]
	\leq -\frac{1}{12} \lambda s^d \alpha^2.
	$$
	Therefore there exists an $N_0 = N_0(s,\lambda,\alpha)$
	such that for all $N \geq N_0$ and $z \in \bbL_{K,s}$,
	$\bbP_{N\lambda}[ H_{N\lambda}(z) \geq (1+\alpha)\lambda s^d N ] \leq e^{-\frac{1}{8} \lambda s^d \alpha^2 N}$
	and
	$\bbP_{N\lambda}[ H_{N\lambda}(z) \leq (1-\alpha)\lambda s^d N ] \leq e^{-\frac{1}{24} \lambda s^d \alpha^2 N}$
	Then \eqref{eq:R-prob} follows from the union bound.
	
	Next, let $\config \in \N_0^{\bbL_{K,s}}$.
	Conditional on $\eta_N$,
	$\Omega(z) \sim \mathrm{Binomial}(H_{N\lambda}(z), 1/N)$
	for all $z \in \bbL_{K,s}$,
	and these components are independent.
	By Lemma~\ref{thm:le-cam},
	$P_N(\Omega(z) = \config(z)) \leq p_{\frac{1}{N}H_{N\lambda}(z)}(\config(z)) + 9/N$.
	Conditional on $R_{s,\alpha,N}$,
	\begin{align*}
		p_{\frac{1}{N}H_{N\lambda}(z)}(\config(z))
		&=
		\frac{1}{\config(z)!} \left(\frac{H_{N\lambda}(z)}{N}\right)^{\config(z)}
		e^{-\frac{H_{N\lambda}(z)}{N}}\\
		&\leq
		\frac{1}{\config(z)!} ( (1+\alpha)\lambda s^d )^{\config(z)}
		e^{-(1-\alpha)\lambda s^d}\\
		&=
		e^{\alpha \lambda s^d} (1 + \alpha)^{\config(z)}
		p_{\lambda s^d}(\config(z)).
	\end{align*}
	Note that the final line is bounded by something which only depends on $\lambda s^d$,
	since $(1 + \alpha)^{n} p_{\lambda s^d}(n) \to 0$ as $n \to \infty$.
	Therefore we get the claimed bound
	\begin{align*}
		P_N(\Omega = \config) = \prod_{z \in \bbL_{K,s}} P_N(\Omega(z) = \config(z))
		&\leq
		\prod_{z \in \bbL_{K,s}} \left(e^{\alpha \lambda s^d} (1 + \alpha)^{\config(z)}
		p_{\lambda s^d}(\config(z)) + \frac{9}{N}\right)\\
		&\leq
		e^{\alpha \lambda K^d} (1+\alpha)^{\sum_{z \in \bbL_{K,s}} \config(z) }
		p_{\lambda s^d}(\config) + \frac{C}{N}.\qedhere
	\end{align*}
\end{proof}

Now we have controlled the number of vertices in each box,
we want to control the edges between the boxes.
First we define a convenient way to sample them.

\begin{definition}
	\label{def:uniform-tensor}
	Given $K$ and $s$ as in Definition~\ref{def:grid},
	let $(U_{i,j}^{w,z} : i,j \in \N, w,z \in \bbL_{K,s}, w \preceq z)$
	be a collection of uniform random variables,
	where $w \preceq z$ means that $w$ precedes $z$ in lexicographic order
	or $w=z$.
	Extend it to a collection
	$(U_{i,j}^{w,z} : i,j \in \N, w,z \in \bbL_{K,s} )$
	by setting $U_{i,j}^{w,z} = U_{j,i}^{z,w}$ if $z \prec w$.
	
	Given $\lambda$ and $N$,
	for each $z \in \bbL_{K,s}$
	label $\eta_N^\zerovec \cap \Lambda_s(z)$ in lexicographic order
	and let $i_x \in \N$ be the index of $x$ within the box $\Lambda_s(\hat{x})$.
	%
	When we construct the edges of $G_g(\eta_N^\zerovec)$,
	if $x, y \in \eta_N^\zerovec$ are distinct vertices,
	let $U_{x,y} = U^{\hat{x},\hat{y}}_{i_x, i_y}$,
	then use $(U_{x,y} : x,y \in \eta_N^\zerovec)$
	to construct the edges as before.
\end{definition}

\begin{remark}
	Suppose we know both $\eta_N$
	and $(U_{i,j}^{w,z} : i,j \in \N, w,z \in \bbL_{K,s})$,
	and want to understand the conditional distribution of the
	\emph{edges} of $\hat{G}_g(\omega_N)$.
	Take two distinct $w, z \in \bbL_{K,s}$,
	suppose we know $\Omega(w) = \Omega(z) = 1$,
	and we want to know whether there is an edge between the two boxes.
	Say the representative of $z$ is $y$ (i.e.\ $y \in \omega_N$ with $\hat y = z$),
	and the representative of $w$ is chosen uniformly from
	the $H_{N\lambda}(w)$ vertices in $\eta_N \cap \Lambda_s(w)$.
	The conditional probability that there is an edge between $y$ and
	the chosen representative of $w$ is
	$\frac{1}{H_{N\lambda}(w)}|
	\{ j \in [ H_{N\lambda}(w) ] : U^{w,z}_{j,i_y} \leq \hat{g}_s(w,z) \} |$.
	As $N \to \infty$,
	this tends to $\hat{g}_s(w,z)$ in probability,
	and it should do so uniformly in $i_y$,
	so just knowing the $U$s gives us little information
	about the edges:\
	the randomness of the edge set comes mostly
	when sampling $\omega_N$ from $\eta_N$.
	
	The following lemma, which we will apply with
	$M = \rdown{N\lambda s^d} \sim H_{N\lambda}(w)$,
	generalises this idea to control the conditional distribution
	of the entire edge set of $\hat{G}_g(\omega_N)$.
\end{remark}

\begin{lemma}[Large deviations for the edges between boxes]
	\label{thm:edge-ldp}
	Suppose $(U^{w,z}_{i,j} : i,j \in \N, w,z \in \bbL_{K,s})$
	is the collection of uniform random variables
	from Definition~\ref{def:uniform-tensor}.
	For $w,z \in \bbL_{K,s}$,
	let $\hat{g}_{w,z} = \hat{g}_s(w,z)$.
	Let $\alpha \in (0,1/2)$.
	Given $M \in \N$, $w \in \bbL_{K,s}$,
	$V = (z_1, \dots, z_{|V|}) \subseteq \bbL_{K,s} \setminus \{w\}$,
	$E \subseteq V$ and
	$\mathbf{k} = (k_{z_1},\dots, k_{z_{|V|}}) \in [\rup{M(1+\alpha)}]^{V}$,
	let $F_{w,V,E,\mathbf{k}}^M$ be the event that
	\begin{align*}
		&\left| \left\{
		j \in [\rup{(1+\alpha)M}] : 
		U^{w,z}_{j,k_z} \leq \hat{g}_{w,z} \, \forall z \in E
		\text{ and }
		U^{w,z}_{j,k_z} > \hat{g}_{w,z} \,\forall z \in V \setminus E
		\right\} \right| \\
		&\mspace{450mu}\leq
		(1+\alpha)^2 M \hat{g}^E_{V,w},
	\end{align*}
	where $\hat{g}^E_{V,w} := (\prod_{z \in E} \hat{g}_{w,z})
	(\prod_{z \in V\setminus E} (1-\hat{g}_{w,z}))$.
	Let
	$$
	F^M := \bigcap_{w \in \bbL_{K,s}}
	\bigcap_{V \subseteq \bbL_{K,s} \setminus \{w\}}
	\bigcap_{E \subseteq V}
	\bigcap_{\mathbf{k} \in [\rup{M(1+\alpha)}]^{|V|}}
	F_{w,V,E,\mathbf{k}}^M.
	$$
	Then there exist constants $c > 0$ and $M_0 \geq 1$
	depending on $K$, $s$, $g$
	and $\alpha$
	such that
	$\bbP(F^M) \geq 1 - \exp(-cM)$
	for all $M \geq M_0$.
\end{lemma}

\begin{proof}
	Note that the displayed cardinality is
	a $\mathrm{Binomial}(\rup{M(1+\alpha)}, \hat{g}^E_{V,w})$ random variable.
	By applying Cram\'{e}r's theorem
	as in the proof of Lemma~\ref{thm:large-deviations},
	we find, for each $w,V,E$,
	that there are constants $c_{w,V,E} > 0$ and $M_{w,V,E} \geq 1$
	such that
	for any $\mathbf{k} \in [\rup{M(1+\alpha)}]^{|V|}$,
	$\bbP[(F_{w,V,E,\mathbf{k}}^M)^c] \leq e^{-c_{w,V,E}M}$
	for all $M \geq M_{w,V,E}$.
	The number of choices for $(w,V,E,\mathbf{k})$ is at most
	$|\bbL_{K,s}| \times 2^{|\bbL_{K,s}|} \times 2^{|\bbL_{K,s}|} \times (2M)^{|\bbL_{K,s}|}
	= |\bbL_{K,s}| (8M)^{|\bbL_{K,s}|}$,
	so by the union bound,
	\begin{align*}
		&\bbP\left(
		\bigcap_{w \in \bbL_{K,s}}
		\bigcap_{V \subseteq \bbL_{K,s} \setminus \{w\}}
		\bigcap_{E \subseteq V}
		\bigcap_{\mathbf{k} \in [\rup{M(1+\alpha)}]^{|V|}}
		F_{w,V,E,\mathbf{k}}^M
		\right)\\
		&\mspace{250mu}\geq
		1 - |\bbL_{K,s}|(8M)^{|\bbL_{K,s}|}\exp\!\left(-M \times \min_{w,V,E} c_{w,V,E} \right)
	\end{align*}
	for all $M \geq \max_{w,V,E} M_{w,V,E}$,
	which implies our claimed bound.
\end{proof}

We will prove Proposition~\ref{thm:transitivity}
by conditioning on the structure of $\hat{G}_s(\omega_N \cup \{x\})$.

\begin{definition}[Set of possible graph structures]
	\label{def:graph-structure}
	Given $K$ and $s$ as in Definition~\ref{def:grid},
	recall that for $x \in \Lambda_K$, $\hat{x}$ is the unique element of $\bbL_{K,s}$
	such that $x \in \Lambda_s(\hat{x})$.
	For any graph $G = (V(G),E(G))$ whose vertex set is a subset of $\Lambda_K$,
	let $F_s(G)$ be the graph with vertex set
	$\{ \hat{x} : x \in V(G) \}$
	and edge set
	$\{ \hat{x}\hat{y} : xy \in E(G) \}$.
	
	Let $x \in \Lambda_K$.
	Let $\mathcal{F}_x$ be the set of all graphs
	whose vertex sets are subsets of $\bbL_{K,s}$
	containing $\hat{x}$.
	For $f \in \mathcal{F}_x$ let $S_{f,x}$ be the event that
	$F_s(\Ghat_s(\omega_N \cup \{x\})) = f$
	and
	$|(\omega_N \cup \{x\}) \cap \Lambda_s(z)| \leq 1$ for all $z \in \bbL_{K,s}$.
	
	Suppose $f$ has vertex set $V(f) = \{z_0, z_1, \dots, z_m\}$
	where $z_0 = \hat{x}$ and edge set $E(f)$.
	Let
	$$
	\hat{g}_s(f)
	:=
	\left(\prod_{\substack{0 \leq k < l \leq m \\ z_k z_l \in E(f)}} \hat{g}_s(z_k,z_l)\right)
	\left(\prod_{\substack{0 \leq k < l \leq m \\ z_k z_l \not\in E(f)}}(1 - \hat{g}_s(z_k,z_l))\right).
	$$
	
\end{definition}

\begin{proof}[Proof of Proposition~\ref{thm:transitivity}]
	By Lemma~\ref{thm:G-coupling}
	applied to the larger box
	$\Lambda_{2K}$,
	for sufficiently small $s$,
	\begin{align}
		\label{eq:mhat}
		\hat{m}_h^{2K,s}(\lambda)
		\leq
		m_h(\lambda) + \frac{\delta}{4},
	\end{align}
	so fix $s$ depending on $\delta$, $K$ and $\lambda$
	such that the above holds, $K/s \in \N$
	and one more condition is satisifed which
	we will specify later.
	Fix $\alpha \in (0,1/2)$
	depending on $\delta$, $K$, $\lambda$ and $s$
	in a way we will also specify later.
	
	Let $x \in \Lambda_K$ and $f \in \mathcal{F}_x$. Then
	\begin{align*}
		P_{N,h}(
		\text{$x$ is connected to a green vertex in $\Ghat_s(\omega_N \cup \{x\})$}
		\,|\,
		S_{f,x}
		)
		&=
		1 - \exp(-h |\cC_{\hat{x}}^f|),
	\end{align*}
	where $\cC^f_{\hat{x}}$ is the component of $\hat{x}$ in $f$.
	We will show that for any $f \in \mathcal{F}_x$,
	with high probability conditional on $\eta_N$
	and $(U^{w,z}_{i,j} : i,j \in \N, w,z \in \bbL_{K,s})$,
	$P_N(S_{f,x})$ is close to the \emph{marginal} probability
	of an analogous event related to a Poisson process $\eta$ of intensity $\lambda$.
	
	Suppose the vertex set of $f$ is
	$V(f) = \{z_0, z_1, \dots, z_{m}\}$
	where $z_0 = \hat{x}$.
	Recall $\Omega(z) := |\omega_N \cap \Lambda_s(z)|$
	for $z \in \bbL_{K,s}$.
	Also define $\Omega^x \in \N_0^{\bbL_{K,s}}$
	by $\Omega^x(z) := |(\omega_N \cup \{x\}) \cap \Lambda_s(z)|$
	for all $z \in \bbL_{K,s}$.
	Let $\config^x_f \in \{0,1\}^{\bbL_{K,s}}$
	with $\config^x_f(z) = 1$ if and only if $z \in V(f)$.
	Let $\config_f(z) = \config^x_f(z)$ for $z \not= \hat{x}$
	and $\config_f(\hat{x}) = 0$.
	Then, conditional on the event $R_{s,\alpha,N}$ from Definition~\ref{def:R-event},
	by Lemma~\ref{thm:large-deviations},
	\begin{align*}
		P_N(\Omega^x = \config^x_f)
		&=
		P_N(\Omega = \config_f)
		\leq
		e^{\alpha \lambda K^d} (1+\alpha)^{(K/s)^d} p_{\lambda s^d}(\config_f)
		+ \frac{C}{N}.
	\end{align*}	
	Conditional on $\Omega^x = \config^x_f$,
	the vertex set of $F_s(\Ghat_s(\omega_N \cup \{x\}))$ is $V(f)$,
	and exactly one uniformly-chosen vertex in each of $\eta_N \cap \Lambda_s(z_1), \dots, \eta_N \cap \Lambda_s(z_m)$
	is open in $\omega_N$.
	We will now show that the probability that $f$
	and $F_s(\Ghat_s(\omega_N \cup \{x\}))$
	have the same edge set is close to $\hat{g}_s(f)$.
	Let the chosen vertex in $\eta_N \cap \Lambda_s(z_i)$ be $X_i$.
	Almost surely $X_0 = x$.
	For $1 \leq k \leq m$,
	the conditional probability, given $(X_0, \dots, X_{k-1})$,
	that the edges between $\{X_0, \dots, X_{k-1}\}$
	and $X_k$ in $\Ghat_s(\omega_N \cup \{x\})$
	match the edges between $\{z_0, \dots, z_{k-1}\}$ and $z_k$ in $f$
	is equal to
	\begin{align}
		\label{eq:inductive-edges}
		\frac{1}{H_{N\lambda}(z_k)}
		\left| \left\{
		j \in [ H_{N\lambda}(z_k) ] :
		\parbox{9cm}{
			\centering
			for all $l \in \{0,\dots,k-1\}$,
			$U^{z_l,z_k}_{i_{X_l},j} \leq \hat{g}_s(z_l,z_k)$
			if $z_lz_k \in E(f)$,
			and
			$U^{z_l,z_k}_{i_{X_l},j} > \hat{g}_s(z_l,z_k)$
			if $z_lz_k \not\in E(f)$
		}
		\right\}\right|.
	\end{align}
	Then, conditional on $R_{s,\alpha,N}$,
	$H_{N\lambda}(z_k)$ is between $N\lambda s^d(1-\alpha)$ and $N\lambda s^d(1+\alpha)$,
	so we apply Lemma~\ref{thm:edge-ldp} with $M = \rdown{N\lambda s^d}$.
	On the event $R_{s,\alpha,N} \cap F^{\rdown{N\lambda s^d}}$,
	\eqref{eq:inductive-edges}
	is bounded by
	$$\left(\prod_{\substack{0 \leq l \leq k-1 \\ z_k z_l \in E(f)}}\hat{g}_s(z_l,z_k)\right)
	\left(\prod_{\substack{0 \leq l \leq k-1 \\ z_k z_l \not\in E(f)}}
	(1-\hat{g}_s(z_l,z_k))
	\right)
	\frac{(1+\alpha)^2}{1-\alpha}.$$
	Therefore, conditional on $R_{s,\alpha,N}$ and $F^{\rdown{N\lambda s^d}}$
	we almost surely have
	\begin{align*}
		P_N(F_s(\Ghat_s(\omega_N \cup \{x\})) = f \,|\, \Omega^x = \config^x_f)
		\leq
		\hat{g}_s(f)
		\left(\frac{(1+\alpha)^2}{1-\alpha}\right)^{m}.
	\end{align*}
	Since $m \leq (K/s)^d$, this gives us a uniform bound of the form
	\begin{align*}
		P_N(F_s(\Ghat_s(\omega_N \cup \{x\})) = f \,|\, \Omega^x = \config^x_f)
		&\leq
		\hat{g}_s(f)
		(1 + C'\alpha),
	\end{align*}
	where the constant $C'$ depends on $K$, $\delta$ and $s$,
	but not on $f$
	and not on $\alpha$.
	
	The other part of the event $S_{f,x}$
	in Definition~\ref{def:graph-structure} is that
	$\Omega^x(z) \leq 1$ for all $z \in \bbL_{K,s}$,
	so we do not map multiple vertices of $\hat{G}_s(\omega_N \cup \{x\})$
	to a single vertex in $F_s(\hat{G}_s(\omega_N \cup \{x\}))$.
	We will show that this part always holds with high probability.
	Since $\alpha < 1/2$, conditional on $R_{s,\alpha,N}$,
	$\Omega$ is stochastically dominated by a collection of $(K/s)^d$ independent
	$\mathrm{Binomial}(\rdown{\frac{3}{2}\lambda s^d N}, 1/N)$ random variables.
	Let $Y_1, \dots, Y_{(K/s)^d}$
	be a family of iid $\mathrm{Poisson}(2\lambda s^d)$ random variables.
	By Lemma~\ref{thm:le-cam}, for sufficiently large $N$,
	conditional on $R_{s,\alpha,N}$,
	\begin{align*}
		P_N\!\left(\max_{z \in \bbL_{K,s}}\Omega^x(z) \geq 2\right)
		&\leq
		\mathbb{P}\!\left( Y_1 \geq 1 \text{ or } \max_{2 \leq i \leq (K/s)^d} Y_i \geq 2 \right)
	\end{align*}
	almost surely.
	Since $\mathbb{P}(Y_1 \geq 1) = 1 - e^{-2\lambda s^d} = 2\lambda s^d + O(s^{2d})$
	and $\mathbb{P}(Y_2 \geq 2) = 2\lambda^2 s^{2d} + O(s^{3d})$,
	by the union bound
	$\mathbb{P}( Y_1 \geq 1 \text{ or } \max_{2 \leq i \leq (K/s)^d} Y_i \geq 2) = O(s^d)$.
	We can and do choose $s$ sufficiently small so that \eqref{eq:mhat} holds
	and $$P_N\!\left(\max_{z \in \bbL_{K,s}}\Omega^x(z) \geq 2\right) < \frac{\delta}{4}$$
	for all sufficiently large $N$.
	
	Now,
	conditional on $R_{s,\alpha,N}$ and $F^{\rdown{N\lambda s^d}}$,
	\begin{align*}
		&P_{N,h}(\text{$x$ is connected to a green vertex in $\Ghat_s(\omega_N \cup \{x\})$})\\
		&\leq
		\sum_{f \in \mathcal{F}_x} P_{N}(S_{f,x}) (1 - e^{-h|\cC_{\hat{x}}^f|})
		+ P_N\!\left( \max_{z \in \bbL_{K,s}}\Omega^x(z) \geq 2 \right)\\
		&\leq
		\sum_{f \in \mathcal{F}_x}
		P_N(F_s(\Ghat_s(\omega_N \cup \{x\})) = f \,|\, \Omega^x = \config^x_f)
		P_N(\Omega^x = \config^x_f)
		(1 - e^{-h|\cC_{\hat{x}}^f|})
		+ \frac{\delta}{4}\\
		&\leq
		e^{\alpha \lambda K^d}(1+\alpha)^{(K/s)^d}(1+C'\alpha)
		\sum_{f \in \mathcal{F}_x}
		\hat{g}_s(f)
		p_{\lambda s^d}(\config_f)
		(1 - e^{-h|\cC_{\hat{x}}^f|})
		+ \frac{C|\mathcal{F}_x|}{N}
		+ \frac{\delta}{4}.
	\end{align*}
	Note that $|\mathcal{F}_x|$ is bounded by a constant depending on $K$ and $s$
	but not on $x$.
	Therefore there exists an $N_0$ depending on $K$, $\delta$, $s$ and $\alpha$
	such that for all $N \geq N_0$,
	$C|\mathcal{F}_x|/N \leq \delta/4$.
	We can and do choose $\alpha = \alpha(\lambda, K,\delta,s) \in (0,1/2)$
	sufficiently small so that
	the above displayed expression is almost surely bounded by
	$$
	\sum_{f \in \mathcal{F}_x}
	\hat{g}_s(f)
	p_{\lambda s^d}(\config_f)
	(1 - e^{-h|\cC_{\hat{x}}^f|})
	+ \frac{3}{4}\delta.
	$$
	
	Let $\eta$ be a Poisson point process of intensity $\lambda$ in $\Lambda_K$.
	Then
	$$\hat{m}_h^{2K,s}(\lambda)
	\geq \bbP_{\lambda,g,h}[
	\text{$x$ is connected to a green vertex in $\hat{G}_s(\eta^x)$}
	].$$
	For $f \in \mathcal{F}_x$, let $\tilde{S}_{f,x}$ be the event that 
	$|\eta^x \cap \Lambda_s(z)| \leq 1$ for all $z \in \bbL_{K,s}$
	and $F_s(\Ghat_s(\eta^x)) = f$.
	Then we have
	\begin{align*}
		\bbP_{\lambda,g,h}[
		\text{$x$ is connected to a green vertex in $\hat{G}_s(\eta^{x})$}
		]
		&\geq
		\sum_{f \in \mathcal{F}_x}\bbP_{\lambda,g}[\tilde{S}_{f,x}](1-e^{-h|\cC_{\hat{x}}^f|})\\
		&=
		\sum_{f \in \mathcal{F}_x}
		\hat{g}_s(f) p_{\lambda s^d}(\config_f) (1 - e^{-h|\cC_{\hat{x}}^f|}).
	\end{align*}
	Therefore, for all $N \geq N_0$,
	we have,
	conditional on $R_{s,\alpha,N}$ and $F^{\rdown{N\lambda s^d}}$,
	\begin{align*}
		P_{N,h}(\text{$x$ is connected to a green vertex in $\Ghat_s(\omega_N \cup \{x\})$})
		&\leq
		\hat{m}_h^{2K,s}(\lambda) + \frac{3}{4}\delta\\
		&\leq
		m_h(\lambda) + \delta
	\end{align*}
	for any $x \in \eta_N^\zerovec$,
	and so
	\begin{align*}
		&\bbP_{N\lambda,g}\!\left[
		P_{N,h}\!\left(
		\text{$x$ is connected to a green vertex in $\hat{G}_s(\omega_N \cup \{x\})$}
		\right)
		> m_h(\lambda) + \delta
		\right]\\
		&\phantom{}
		\leq
		\bbP_{N\lambda}[R_{s,\alpha,N}^c]
		+ \bbP_{N\lambda,g}[(F^{\rdown{N\lambda s^d}})^c].
	\end{align*}
	By Lemma~\ref{thm:large-deviations} and Lemma~\ref{thm:edge-ldp},
	this bound decays exponentially in $N$,
	uniformly in $x$.
\end{proof}

\subsection*{Acknowledgements}
Thanks to Matt Roberts for suggesting this problem and for many helpful discussions. Thanks also to Luca Zanetti for comments on a draft.
I would also like to thank two anonymous reviewers for pointing out
that extending the main theorem
from bounded to unbounded
connection functions
could be done in one paragraph
instead of three pages,
and catching some important mistakes.

\bibliographystyle{plainurl}
\bibliography{sharpness-references}

\end{document}